\documentclass{amsart}
\usepackage{graphicx}
\usepackage{amssymb,amscd,amsthm,amsxtra}
\usepackage{hyperref}
\usepackage{latexsym}
\usepackage{epsfig}
\usepackage{mathtools}
\usepackage{esint}
\usepackage{color}

\newtheorem{thm}{Theorem}[section]

\newtheorem{lem}[thm]{Lemma}

\theoremstyle{definition}

\theoremstyle{remark}

\numberwithin{equation}{section}
\newcommand{\R}{\mathbb R}

\newcommand{\p}{\partial}

\newcommand{\comment}[1]{}
\begin{document}

\title[Boundary Harnack Principle]{On the Boundary Harnack Principle in H\"older domains}
\author{D. De Silva}
\address{Department of Mathematics, Barnard College, Columbia University, New York, NY 10027}
\email{\tt  desilva@math.columbia.edu}
\author{O. Savin}
\address{Department of Mathematics, Columbia University, New York, NY 10027}\email{\tt  savin@math.columbia.edu}
\begin{abstract}We investigate the Boundary Harnack Principle in H\"older domains of exponent $\alpha>0$ by the analytical method developed in \cite{DS}.
 \end{abstract}

%\thanks{}
%\subjclass{}%
%\keywords{One-phase free boundary problem; Harnack Inequality}
\maketitle

\section{Introduction}

In this paper we continue the study of the boundary Harnack principle for solutions to elliptic equations, based on the method developed in \cite{DS}. 
The classical boundary Harnack principle 
states that two positive harmonic functions that vanish on a portion of the boundary of a Lipschitz domain must be comparable up to a multiplicative constant, see for example \cite{A,D,K,W}. 
Further extensions to more general operators and more general domains were obtained in several subsequent works \cite{Kr,CFMS,JK,BB1, BBB, KiS}.

In particular, Bass and Burdzy \cite{BB1,BB2} and Banuelos, Bass and Burdzy \cite{BBB} provided sharp versions using probabilistic methods. They established the boundary Harnack principle for nondivergence elliptic operators in H\"older domains (or more general twisted H\"older domains) of exponent $\alpha > \frac 12$, and for divergence operators in H\"older domains of arbitrary exponent $\alpha >0$. For the case of divergence operators, an analytical proof based on Green's function was given by Ferrari in \cite{F}. 

%We now state precisely one of this results.

To state precisely the boundary Harnack principle in H\"older domains, first we introduce some notation. Let $g: \overline {B}_1' \to \R$ be a $C^\alpha$ H\"older function of $n-1$ variables with $g(0)=0$, and $\alpha \in (0,1)$. Denote by $\Gamma \subset \R^n$ the graph of $g$,
$$\Gamma:=\{x_n=g(x') \}, \quad \quad 0 \in \Gamma,$$
and by $\mathcal C_r$ the cylinder above $B_r'$ and at height $r$ on top of $\Gamma$ 
 $$\mathcal C_r:=\{ x' \in B'_r , \quad g(x')<x_n < g(x')+r\}.$$
We say that $\mathcal C_1$ is a $C^\alpha$-Holder domain in a neighborhood of $\Gamma$.

The following version of the boundary Harnack principle is due to Banuelos, Bass and Burdzy, see \cite{BBB}.
\begin{thm}\label{T1}
Let $Lu=div (A(x) \nabla u)$ be a uniformly elliptic linear operator, and assume that $u, v$ are two positive solutions to 
$$Lu=Lv=0, \quad \mbox{in $\mathcal C_1$,} $$ 
which vanish on $\Gamma$. Then
$$ \frac{u}{v} \ (x) \le C \, \,  \frac u v \left (\frac 12 e_n \right) \quad \mbox{in $\mathcal C_{1/2}$},$$
with $C$ depending on $n$, $\alpha$, $\|g\|_{C^\alpha}$ and the ellipticity constants of $L$. 
\end{thm}

The assumption that $u=v=0$ in $\Gamma$ is understood in the $H^1$ sense, i.e. $u, v \in H_{0}^1(\mathcal C_1)$ in a neighborhood of $\Gamma$. %In \cite{BBB} the symmetry of the 
%Green function was used and the coefficient matrix $A$ was assumed to be symmetric, however we show this is not necessary.  

Recently, in \cite{DS} we found a direct analytical method of proof of the boundary Harnack principle based on an iteration scheme and Harnack inequality. In particular we established the corresponding results in H\"older domains of exponent $\alpha > \frac 12$ for general equations either in divergence or nondivergence form.

In the present paper we discuss further the case of H\"older domains of arbitrary exponent $\alpha >0$, and give a proof of Theorem \ref{T1} using the same ideas from \cite{DS}. We also consider some extensions of Theorem \ref{T1} to non-divergence equations whose coefficients remain constant in the vertical direction. 
%We now state precisely one of this results.

The paper is self-contained and is organized as follows. In Section 2 we give two lemmas concerning Harnack inequality outside domains of small capacity. In Section 3 we use these lemmas and employ the arguments from \cite{DS} to prove Theorem \ref{T1}. Finally in Section 4 we provide some extensions of Theorem \ref{T1} to more general divergence operators, and certain non-divergence or fully nonlinear operators.

\section{Two lemmas}

In this section we present two lemmas concerning solutions to divergence equations in domains whose complement in the unit cube $Q_1$ has small capacity.  

Given a domain $\Omega$ and a compact set $K \subset \Omega$, we say that two functions $u,v \in H^1(\Omega)$ agree on $K$, and write $u=v$ in $K,$ if $u-v \in H_{0,loc}^1(K^c)$. Here $K^c$ denotes the complement of $K$ in $\R^n.$

In particular, if $L$ is a uniformly elliptic operator in divergence form $$Lu= div (A(x) \nabla u),$$
with $$\mbox{$A(x)$ measurable,} \quad \quad \Lambda \,  |\xi|^2 \ge \xi^T \, A(x) \, \xi \ge \lambda \, |\xi|^2, \quad \quad \lambda>0,$$
then $u$ solves 
\begin{equation}\label{Lu}
Lu=0 \quad \mbox{in $\Omega \setminus K$,} \quad \mbox{and $u=0$ on $\p \Omega$, $u=1$ in $K$,} 
\end{equation}
means that $Lu=0$ in the open set $\Omega \setminus K$, and  
$$u-\eta \in H_0^{1} \left(\Omega \setminus K\right),$$
where $\eta \in C_0^\infty(\Omega),$ and $\eta=1$ in a neighborhood of $K$. 
 
 Notice that the solution $u$ to \eqref{Lu} is a supersolution in $\Omega$, i.e. $Lu \le 0$ in $\Omega$.

Let $Q_1$ denote the unit cube in $\R^n$ centered at $0$, and $E \subset Q_1$ a closed set. 
Set, 
$$cap_{3/4}(E):=cap_{Q_1}(E \cap Q_{3/4})=  \inf_{w \in \mathcal A}\int_{Q_1} |\nabla w|^2 dx,$$
where
$$\mathcal A:=\{w \in H_0^1(Q_1),  \quad w=1 \quad \mbox{in $E\cap \overline Q_{3/4}$} \}.$$ 

The first lemma states that a solution to $Lv=0$ in $Q_1 \setminus E$ satisfies the Harnack inequality in measure if $E$ has small capacity. Positive 
constants depending on the dimension $n$ and the ellipticity constants $\lambda$, $\Lambda$ are called universal. 

\begin{lem}\label{l1}
Assume $v \ge 0$ is defined in $Q_1 \setminus E$ and satisfies
$$ Lv =0.$$Let $$Q^i:=Q_{1/8}(x_i) \subset Q_{1/2}, \quad i=1,2$$ be two cubes of size $1/8$ included in $Q_{1/2}$. Assume that 
$$cap_{3/4}(E) \le \delta \quad \mbox{and} \quad \frac{|\{v \ge 1\} \cap Q^1|}{|Q^1|} \ge 1/2,$$
for some $\delta$ small, universal. 
Then $$\frac{|\{v \ge c_0\} \cap Q^2|}{|Q^2|} \ge 1/2$$
for some $c_0$ small.
\end{lem}

The second lemma is standard and states that the weak Harnack inequality holds for a subsolution $v\ge 0$ which vanishes on a set $E$ of positive capacity.
\begin{lem}\label{l2}
Assume that $v \ge 0$ in $Q_1$, and 
$$\mbox{$Lv \ge 0$ in $Q_1$, $v=0$ in $E \cap \overline Q_{3/4}$.}$$ If
$$ cap_{3/4}(E) \ge \delta,$$ then
$$v(0) \le (1-c(\delta)) \|v\|_{L^\infty}.$$

\end{lem}

\begin{proof}[Proof of Lemma \ref{l1}] Let $\psi$ be the solution to 
$$\mbox{$L \psi=0$ in $Q_{3/4} \setminus K$, \quad \quad $\psi =1$ in $K$, \quad $\psi=0$ on $\p Q_{3/4}$,}$$
where $K$ is a compact subset of $\{v \ge 1\} \cap (Q^1 \setminus E)$ with $|K| \ge \frac 14 |Q^1|$ (see Figure \ref{fig1}).
By hypothesis and weak Harnack inequality we find 
\begin{equation}\label{A1}
\psi \ge c_0 \quad \mbox{ in} \quad  Q_{1/2},
\end{equation} 
for some small $c_0$.

\begin{figure}[h] 
\includegraphics[width=0.6 \textwidth]{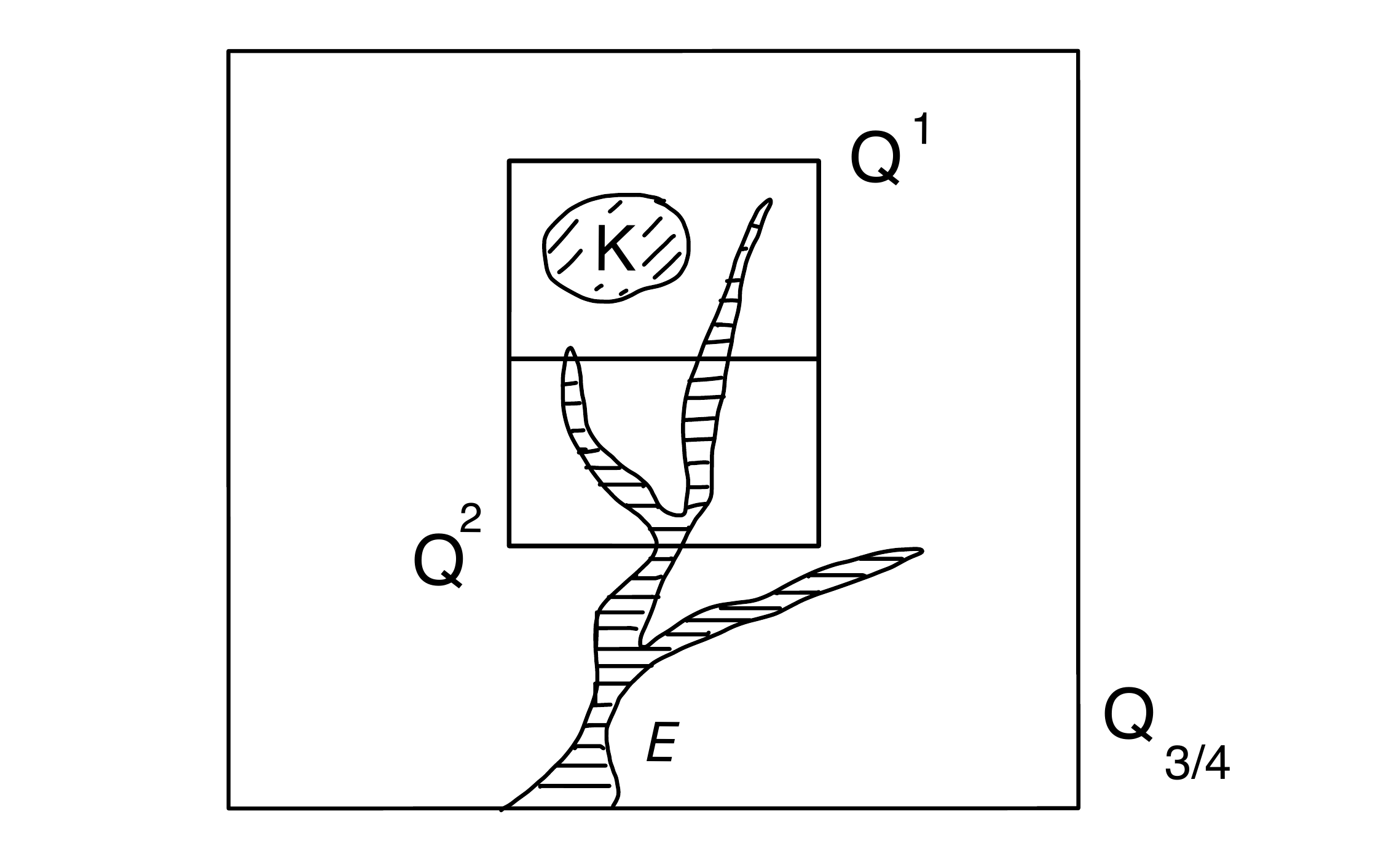}
\caption{ }
   \label{fig1}
\end{figure}
Similarly as above we define $\phi$ to be the solution to
\begin{equation}\label{A1.5}
\mbox{$L \phi=0$ in $Q_1 \setminus (E \cap Q_{3/4}) $,  \quad \quad $\phi =1$ in $E \cap Q_{3/4}$, $\phi=0$ on $\p Q_1$.}
\end{equation}
We claim that if $\delta$ is chosen sufficiently small then,
\begin{equation}\label{A2}
|\{\phi > \frac 14 c_0  \} \cap Q^2| \le \frac 12 |Q^2|.
\end{equation} 
For this we let $w$ be the solution to \eqref{A1.5} when $L=\triangle$. The Dirichlet energies of $\phi$ and $w$ are comparable since
$$\int (\nabla (\phi-w))^T A \nabla \phi \, dx =0 ,$$
hence 
$$cap_{3/4} (E) \le \int |\nabla \phi |^2 \,dx \le C \int (\nabla \phi)^T A \nabla \phi dx \le C \int |\nabla w |^2 \,dx= C cap_{3/4} (E).$$
By Poincar\'e inequality we find
$$\int \phi^2 \,dx \le C \int |\nabla \phi |^2 \,dx\le C \delta,$$ which gives the claim \eqref{A2}.

Next we compare $2v$ with $\psi-\phi$ in $Q_{3/4} \setminus E$. 

They satisfy the same equation in $Q_{3/4} \setminus (E \cup K)$, and in a neighborhood of $K$ by the continuity of $v$ we have
$$2v \ge 1 \ge  \psi-\phi.$$
On the other hand $$\psi-\phi \le 0 \quad \mbox{on} \quad \p (Q_{3/4} \setminus E)$$ in the sense that $(\psi -\phi)^+ \in H_0^1 (Q_{3/4} \setminus E)$. Since $v \ge 0$, the maximum principle gives
$$2v \ge \psi -\phi,$$
which by \eqref{A1},\eqref{A2} yields the desired conclusion.

\end{proof}

\begin{proof}[Proof of Lemma \ref{l2}]
Assume that $\|v\|_{L^\infty}=1$. Then, by the maximum principle we have $$1-v \ge \phi,$$ with $\phi$ as in \eqref{A1.5} above. It suffices to show that $\phi \ge c(\delta)$ on $\p Q_{7/8}$ which by the maximum principle implies the desired conclusion $\phi(0) \ge c(\delta)$ small. Since all the values of $\phi$ are comparable near $\p Q_{7/8}$ by the Harnack inequality, we need to show that $\phi \ge c'(\delta)$ at some point on $\p Q_{7/8}$.

Assume by contradiction that $|\phi| \le \mu$ is very close to $0$ on $\p Q_{7/8}$. The Caccioppoli inequality (we think that $\phi$ is extended to $0$ outside $Q_1$) implies
\begin{equation}\label{A3}
 \| \nabla \phi\|_{L^2 (Q_1 \setminus Q_{15/16})} \le C \|\phi \|_{L^2(Q_1\setminus Q_{7/8})} \le C \mu.
 \end{equation}
On the other hand if $\eta \in C_0^\infty(Q_1)$ with $\eta=1$ in $Q_{15/16}$ then 
$$\int \nabla [\eta^2 (1-\phi)]  A \nabla \phi \,dx=0,$$
hence
$$\int \eta^2 \nabla \phi  A \nabla \phi \,dx \le C \int |\nabla \phi|^2 |\nabla \eta|^2 \,dx\le C \mu^2.$$
This together with \eqref{A3} implies that the Dirichlet energy of $\phi$ in $Q_1$ is bounded above by $C \mu^2$, and we reach a contradiction.  
\end{proof}

\section {Proof of Theorem \ref{T1}}

This section is devoted to the proof of Theorem \ref{T1}. We recall that $\Gamma$ denotes the graph of a $C^\alpha$ function $g$, with $\alpha \in (0,1)$, 
$$\Gamma:=\{x_n=g(x') \}, \quad \quad 0 \in \Gamma,$$
and $\mathcal C_r$ denotes the cylinders of size $r$ on top of $\Gamma$ 
 $$\mathcal C_r:=\{ x' \in B'_r , \quad g(x')<x_n < g(x')+r\}.$$
The main idea of the proof is to show through an iterative procedure that a solution $w$ which vanishes on $\Gamma$ and is mostly positive in $\mathcal C_r$, becomes 
positive near the origin.  

We denote by 
$$\mathcal A_{r}:=\left\{x \in B_r'| \quad g(x')+r^\beta \le x_n < g(x') + r \right \}, $$
the points in the cylinder $\mathcal C_r$ at height greater than $r^\beta$ on top of $\Gamma$, for some $\beta>1$.

We divide the proof in three steps.

\

\textit{Step 1.}  We show that,
there exist $C_0, \beta>1$ depending on $n$, $\alpha$, $\|g\|_{C^\alpha}$, and the ellipticity constants of $L$, such that if $w$ is a solution to $ L w=0$ in $\mathcal C_r$ (possibly changing sign) which vanishes on $\Gamma$,
$$w \geq f(r) \quad \quad \text{on} \quad \mathcal A_{r}, $$
and
$$w \geq -1 \quad \text{on $\mathcal C_r$},$$
where $$f(r):= e^{C_0r^\gamma}, \quad \gamma:=\beta(1-\frac{1}{\alpha}) <0,$$
then,
\begin{equation}\label{301}
w \geq f( \frac r 2) \, \, a \quad \text{on $\mathcal A_{\frac r 2}$},
\end{equation} and
\begin{equation}\label{302}
w \geq -a \quad \text{on $\mathcal C_{\frac r 2}$}, 
\end{equation}
for some small $a=a(r)>0$, as long as $r\leq r_0$ universal.

The conclusion can be iterated and we obtain that if the hypotheses are satisfied in $\mathcal C_{r_0}$ then 
$$w > 0 \quad \text{on the line segment $\{te_n, \quad 0<t<r_0\}.$}$$

Since $g$ is H\"older continuous, we can apply interior Harnack inequality to $w+1$ in a chain of balls and need $$C(r^\beta)^{1-\frac 1 \alpha} =C r^\gamma \quad \mbox{balls}$$ to connect a point in $\mathcal A_{r/2}$ with a point in $\mathcal A_r$. We conclude that
\begin{equation}\label{w1}w \geq (f(r)+1) e^{-C_1r^\gamma} -1 \quad \text{in $\mathcal A_{r/2}$},\end{equation}
for some $C_1$ universal, hence $w \ge 1$ in $\mathcal A_{r/2}$ if $C_0$ is sufficiently large.

Next we take a point on $\Gamma:=\{x_n=g(x')\}$, say $0$ for simplicity, 
and consider the cubes of size $r^{\beta/\alpha}$ centered on the $e_n$ axis, i.e $Q_{r^{\beta / \alpha}}(t e_n)$ (see Figure \ref{fig2}). 

When $t>Cr^\beta$ the cube is in the interior of the domain and when $t<-C r^\beta$ the cube is in the complement. 
There are $C r^\gamma$ stacked cubes which connect the domain with its complement. 
The graph property of the domain implies that the capacity of the complement 
$$E=\{x_n \le g(x')\}$$ in $Q_{r^{\beta / \alpha}}(t e_n)$ is decreasing with $t$. 
By continuity we can find a cube centered at $t_0 e_n$ such that, after a rescaling of factor $r^{-\beta/\alpha}$, $cap_{3/4}(E)=\delta$ in that cube, with $\delta$ as in Lemma \ref{l1}.  

\begin{figure}[h]
\includegraphics[width=0.7 \textwidth]{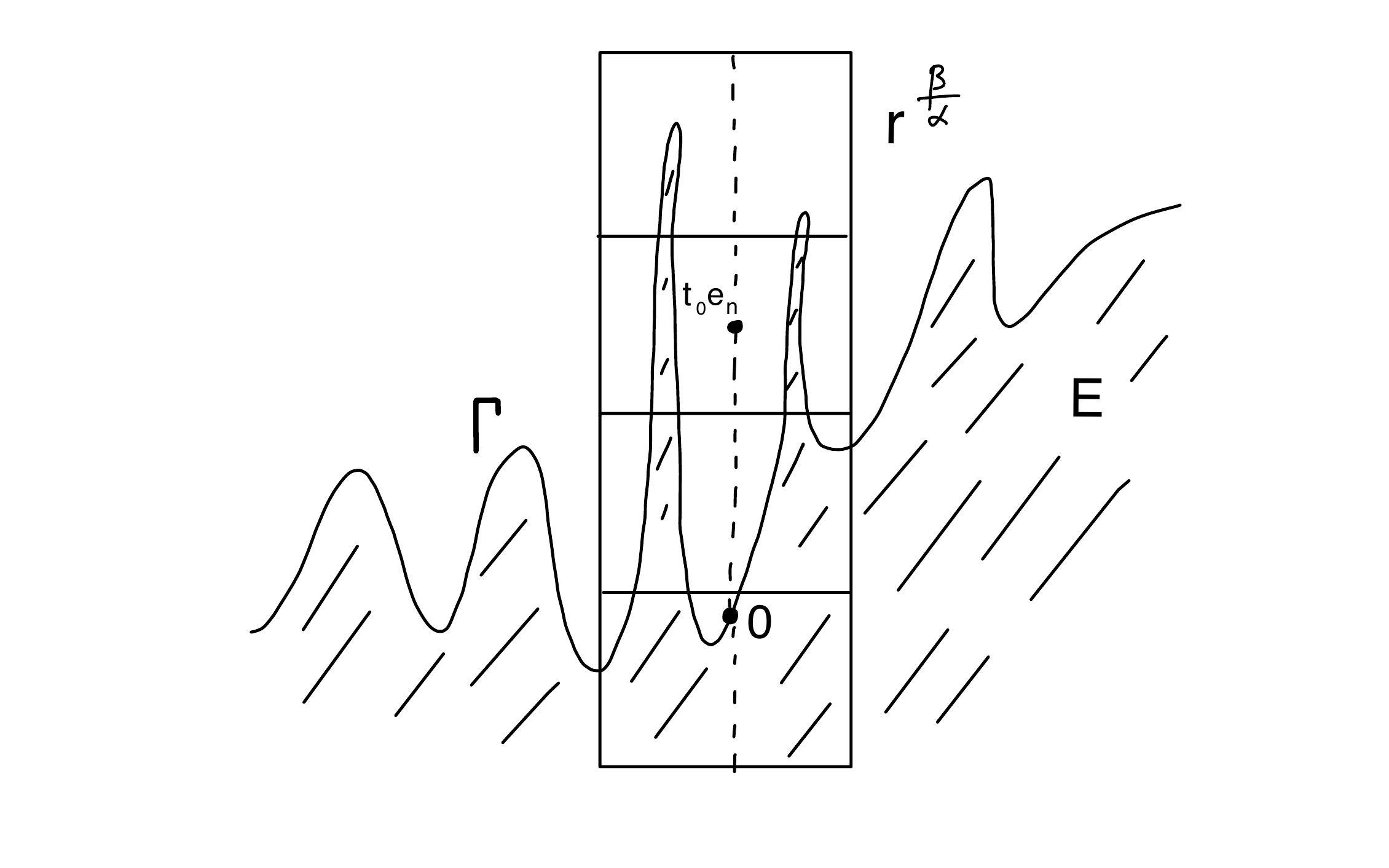}
\caption{}
    \label{fig2}
\end{figure}

For all cubes centered at $te_n$ with $t \ge t_0$ we can apply Lemma \ref{l1} repeatedly for $w+1$ and obtain an inequality in measure as in \eqref{w1},
$$|\{w \ge 1\} \cap Q^i(te_n)| \ge \frac 12 |Q^i|, \quad \quad Q^i(te_n):= Q_{\frac 18r^{\beta/\alpha}}(t e_n).$$
Thus $\{w^-=0\}$ has positive density in all cubes centered at $te_n$ with $t \ge t_0$.

Now we notice that we can apply weak Harnack inequality for $w^-$ in all cubes, in the top cubes with $t\ge t_0$ because of the density property, and in the bottom ones with $t \le t_0$ because of Lemma \ref{l2}. 

Hence $w^-$ decreased by a fixed factor on the $e_n$ axis passing through a point on $\Gamma$, with respect to its maximum over all cubes of size $r^{\beta/\alpha}$ centered on that axis. As we move each time a $r^{\beta/\alpha}$ distance inside the domain from the sides of $\mathcal C_r$, $\sup w^-$ decreases geometrically hence, 
$$w \geq - e^{-c_0 r^{1-\beta/\alpha}} \quad \text{in $\mathcal C_{r/2}$},$$
for $c_0$ small universal.
We choose 
$$a(r):=e^{-c_0 r^{1-\beta/\alpha}},$$
and in view of \eqref{w1}, our claim 
$$w \ge 1 \ge a(r) f(\frac r2),$$is satisfied for all $r$ small. 

\

\textit{Step 2.} [Carleson estimate] We show that,
$$u,v \leq C_2 \quad \text{in $\mathcal C_{1/2}$,}$$
with $C_2$ universal. We apply an iterative argument similar to the one in Step 1. 
Since $u(e_n/2)=1$, the interior Harnack inequality gives that
\begin{equation}\label{int}u \leq e^{\, C_1 h_\Gamma^{1-1/\alpha}} \quad \quad \text{in $\mathcal C_{3/4}$}, \quad \quad h_\Gamma(x):=x_n-g(x'),\end{equation}
 with $C_1$ universal.
With the same notation as Step 1, we wish to prove that if $r$ is smaller than a universal $r_0$ and
$$u(y) \geq f(r),$$ for some $y \in \mathcal C_{1/2}$,
then we can find $$z \in S:=\{|y'-z'|=r, \quad 0<h_{\Gamma}(z)< r^\beta\},$$ such that
$$u(z) \geq f(\frac{r}{2}).$$
Since $|z-y|\le C r^\alpha$, we see that for $r$ small enough, we can build a convergent sequence of points $y_k \in \mathcal C_{3/4}$ with $u(y_k)\ge f(2^{-k}r )\to \infty$. On the other hand the extension of $u$ by $0$ below $\Gamma$ is a subsolution in a neighborhood of $\Gamma$. Therefore $u$ is bounded above, and we reach a contradiction.

To show the existence of the point $z$ we let 
$$w:=(u- \frac 12 e^{C_0 r^\gamma})^+, \quad \quad \mbox{with} \quad C_0 \gg C_1.$$ 
By \eqref{int} we know that 
$$w=0 \quad \mbox{ when} \quad h_\Gamma(x) \ge r^\beta.$$ By Lemma \ref{l1} this estimate can be extended in measure for the cubes of size $r^{\beta/\alpha}$ with $t \ge t_0$ since the capacity of the complement is bounded above. More precisely, as in Step 1, in each cube of size $r^{\beta/\alpha}$ we have that either $\{w=0\}$ has positive density (for the cubes with $t \ge t_0$), or positive capacity (for the cubes with $t \le t_0$). 

Moreover, if our claim is not satisfied then we apply Weak Harnack inequality for $w$ repeatedly as in Step 1 above. As we move inside the domain from the sides of $\mathcal C_r(y',g(y'))$ we obtain
$$w \leq f(\frac r2) \, \, e^{-c_0 r^{1-\beta/\alpha}} \quad \text{in $\mathcal C_{r/2}(y',g(y'))$}.$$
In particular $$\frac 12 f(r) \le w(y) \le f(\frac r2) e^{-c_0 r^{1-\beta/\alpha}},$$
and we reach a contradiction.

\

\textit{Step 3.} We prove the theorem using the Steps 1 and 2 above. After multiplication by a constant we may assume that $u=v=1$ at $\frac 12 e_n$. It suffices to show that for a large constant $C_3>0$ universal, 
$$w:=C_3 u - C_3^{-1} v \geq 0 \quad \text{in $\mathcal C_{1/2}.$}$$
By Step 2 we know that $v \leq C_3$ hence $w \ge -1$ in $\mathcal C_{3/4}.$ Moreover, since $u(e_n/2)=1$, we conclude by interior Harnack for $u$ that
$$w \geq f(r_0)\quad \text{in} \quad \mathcal C_{3/4} \cap \{x_n \ge r^\beta_0\},$$ provided that $C_3$ is chosen sufficiently large. Here $f$, $r_0$ and $\beta$ are as in Step 1.

We conclude by Step 1 that $w \geq 0$ on the line $\{t e_n, 0<t<3/4\}$. We can repeat the argument at all points on $\Gamma \cap \overline {\mathcal C}_{1/2}$, and the theorem is proved.
\qed

\section{Some extensions}

In this section we state a few variants of the Theorem \ref{T1}. First we remark that the proof applies to operators involving lower order 
terms.

\begin{thm}
The statement of Theorem $\ref{T1}$ holds for general uniformly elliptic operators
$$ Lu=div ( A(x) \nabla u ) + b(x) \cdot \nabla u + d(x) u, \quad \quad b \in L^q, \, d \in L^{q/2}, \quad q>n,$$
with the constant $C$ depending also on $q$, $\|b\|_{L^q}$, $\|d\|_{L^{q/2}}$.
\end{thm} 

Indeed, we only need to check that the statements of Section 2 continue to hold in the small cubes of size $r^{\beta/\alpha}$. 
After a dilation this corresponds to proving Lemmas \ref{l1} and \ref{l2} for operators $L$ as above with $\|b\|_{L^q}$, $\|d\|_{L^{q/2}}$ sufficiently small. 
The proofs are identical since the presence of such lower order terms does not affect the energy estimates.

A counterexample of Bass and Burdzy in \cite{BB2} shows that Theorem \ref{T1} does not hold in general for nondivergence equations when $\alpha < \frac 12$. Here we remark that Theorem \ref{T1} remains valid with $\alpha>0$ for nondivergence linear operators which are translation invariant 
in the vertical direction.

\begin{thm}
The statement of Theorem $\ref{T1}$ holds for linear nondivergence uniformly elliptic operators of the form
$$ Lu= tr \, \, A(x') D^2 u.$$
\end{thm} 

In this theorem we assume for simplicity that the coefficient matrix $A$ depends continuously on its argument. Since $u$, $v$ might not be continuous at all points on $\Gamma$, the hypothesis that $u$, $v$ vanish on the boundary is understood in the sense 
that their 
extensions with $0$ below $\Gamma$ are bounded subsolutions for $L$, see \cite{DS}.

 In this case we provide the corresponding lemmas of Section 2 by defining the capacity (with respect to $L$) as
 $$cap_{3/4}(E)=\inf_{Q_{1/4}} \phi,$$
 where $\phi$ solves
 $$\mbox{$L \phi=0$ in $Q_1 \setminus (E \cap Q_{3/4}) $, $\phi =1$ in $E \cap Q_{3/4}$, $\phi=0$ on $\p Q_1$.}$$
  Then Lemma \ref{l2} follows directly from the definition of the capacity, with $c(\delta) = \delta$. 
  For Lemma \ref{l1} we see that \eqref{A2} is satisfied since by the Weak Harnack inequality 
  the set $\{\phi>c_0\}$ must have small measure in $Q_1$ if $\delta$ is sufficiently small. The rest of the proof is the same.

The arguments of Section 3 can be repeated in the same way. The invariance of the operator $L$ with respect to the vertical direction and the 
graph property of the boundary imply that the capacity of the complement $E$ in the cubes $Q_{r^{\beta / \alpha}}(t e_n)$ is monotone in $t$. 
We can apply again Lemma \ref{l1} for the top cubes with $t \ge t_0$ and Lemma \ref{l2} for the bottom cubes with $t \le t_0$, and carry on as before.

\

We also discuss the case of fully nonlinear operators
\begin{equation}\label{F}
F(D^2u) =0 \quad \mbox{in} \quad \mathcal C_1,
\end{equation}
with $F$ uniformly elliptic with constants $\lambda, \Lambda$, and homogenous of degree 1. 

We can prove the lemmas of Section 2 for the operator $F$ by using as capacity the definition 
above with $L\phi=F(-D^2\phi)$. Then Lemma \ref{l2} follows again directly from the definition. 
For the proof of Lemma \ref{l1} we choose the function $\psi$ to satisfy $\mathcal M^-_{\lambda/n,\Lambda}(\psi) =0$. Here as usual, $\mathcal M^-$ denotes the extremal Pucci operator, $$\mathcal M^-_{\lambda, \Lambda}(M)= \inf_{A} tr(AM),$$ with $A$ a symmetric matrix whose eigenvalues belong to $[\lambda, \Lambda].$

Then $\psi-\phi$ is a subsolution
  $$ F(D^2(\psi-\phi)) \ge F(D^2(-\phi)) + \mathcal M^-_{\lambda/n,\Lambda}(\psi)  \ge 0,$$
  and the rest of proof remains as before. However in the proof of the main Theorem \ref{T1} only Step 2, the Carleson estimate, can be carried out in this setting,
 since for Step 1 we need the lemmas of Section 2 to hold not only for solutions of the operator $F$ but for the difference of two such solutions as well. 

  \begin{thm}[Carleson estimate]
  Assume that $u \ge 0$ satisfies \eqref{F} and $u$ vanishes on $\Gamma$. Then
  $$ u \le C u(\frac 12 e_n) \quad \mbox{in} \quad \mathcal C_{1/2},$$
  with $C$ depending on $n$, $\alpha$, $\|g\|_{C^\alpha}, \lambda$ and $\Lambda$. 
  \end{thm}
  
  Finally we mention that in $\R^2$ Theorem \ref{T1} holds under very mild assumptions on the domain and the operator. Here we state a version for $L^\infty$ graphs and linear operators. 
  
  \begin{thm}
 Assume $\Gamma \subset \R^2$ is the closure of the graph of a function $g$ with $\|g\|_{L^\infty} \le 1/4$. Then the statement of Theorem \ref{T1} 
 holds for uniformly elliptic linear operators $L$ in divergence or nondivergence form with constant $C$ depending only on the ellipticity constants of $L$.
  \end{thm}
  
  We only sketch Steps 1 and 2 of Section 3 in this setting which can be adapted to more general situations. They are based on topological considerations and do not require an iterative argument.  
  
  \
  
 {\it Step 1.} If $Lw=0$ and $w \ge -1$ in $\mathcal C_1$, and $w$ vanishes continuously on $\Gamma$, then $w>0$ in $\mathcal C_{1/2}$ provided 
 that $u (\frac 12 e_2)$ is large. 
 
 \begin{figure}[h]
\includegraphics[width=0.6 \textwidth]{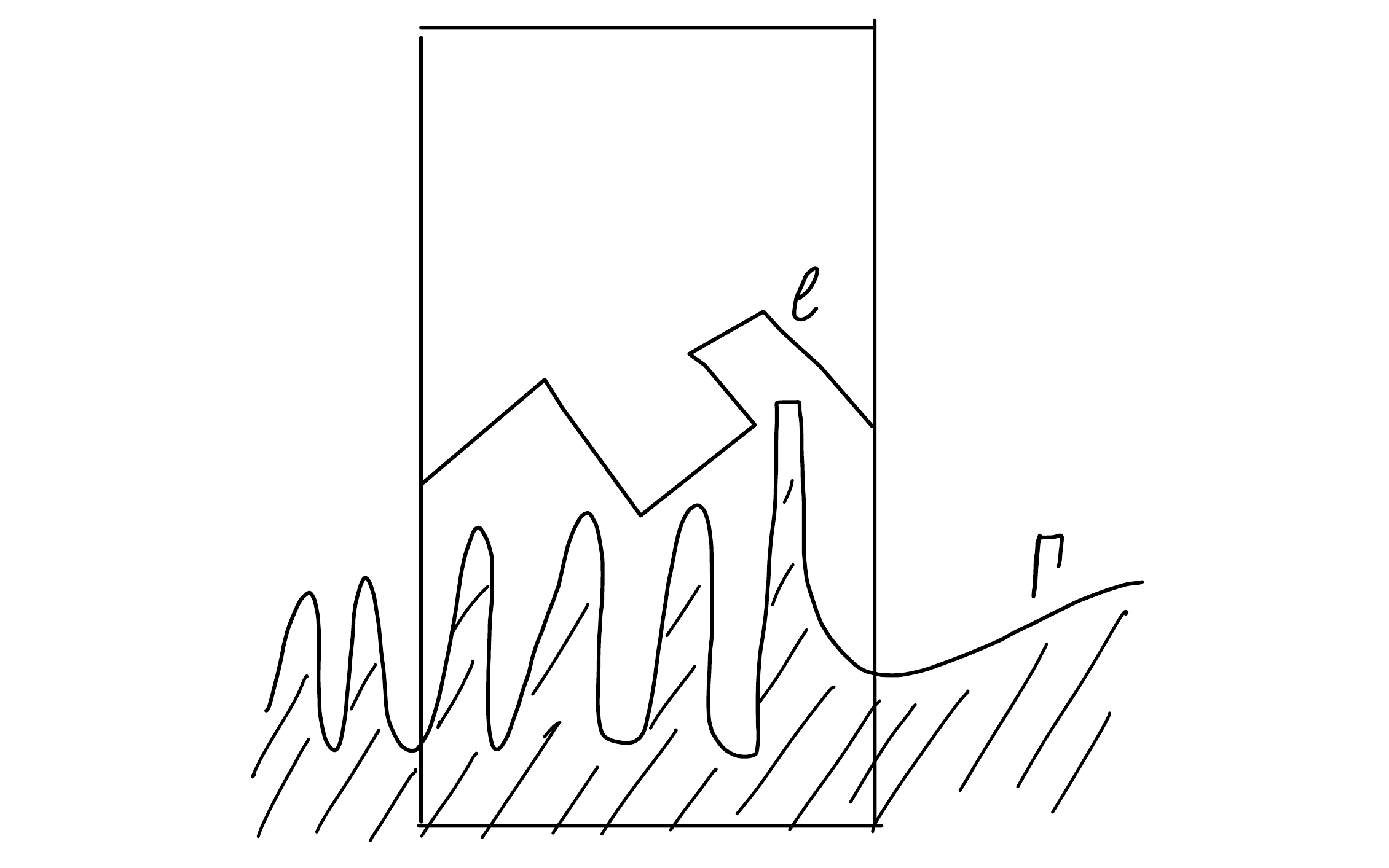}
\caption{}
    \label{fig3}
\end{figure}

To prove this, we assume by contradiction that there is a connected component $U$ of $\{w<0\}$ which intersects $\mathcal C_{1/2}$. 
This connected component must exit $\mathcal C_1$ and since $u(\frac 12 e_2) \gg1$, $U$ must stay close to $\Gamma$. 
Thus we can find a nonintersecting polygonal line $\ell$, say included in $$l \, \subset \, \, \left\{\frac 12 \le x_1 \le \frac 34\right\} \cap U,$$ which connects the two lateral 
sides $x_1=\frac 12$ and $x_1= \frac 34$ (see Figure \ref{fig3}). The line $\ell$ splits the cylinder $$D:=(\frac 12, \frac 34) \times (\frac 12, \frac 12)$$ into two disjoint sets, and we define $\tilde w$ to be equal to $w$ on the set ``above" $\ell$ and $\tilde w=\min \{ w, 0\}$ on the set ``below'' $\ell$. Then $\tilde w$ is a supersolution of $L$ in $D$. Since $\tilde w$ is sufficiently large in a ball above $\ell$, and $\tilde w \ge -1$ in $D$ we find that $\tilde w \ge 0$ on the segment 
$$ \{x_1=\frac 5 8\} \cap \{|x_2| \le \frac 38\}.$$
We reached a contradiction at the point where $\ell$ intersects this segment.

\

{\it Step 2.} Assume $Lu=0$ and $u \ge 0$ in $\mathcal C_1$, and $u$ vanishes continuously on $\Gamma$, with $u(\frac 12 e_2)=1$. 
Then $u \le C$ in $\mathcal C_{1/2}$, for some large $C$. 

\

This follows similarly as in Step 1. If $\{u>C\}$ has a connected component that intersects $\mathcal C_{1/2},$ then we can find a polygonal line $\ell$ 
as above where $u$ is large. Thus $\min\{u,C\}$ extended by $C$ below $\ell$ is a supersolution for $L$ in $D$, and the maximum principle implies that $u$ 
 is large at the point $(5/8,1/2)$. Therefore by Harnack inequality $u(0,1/2)$ is large as well, and we reach a contradiction.

\end{document}